\documentclass[12pt,a4paper]{amsart}
\usepackage{amssymb}
\usepackage{dsfont}
\usepackage{amscd}
\usepackage[mathscr]{euscript} 
\usepackage[all]{xy}
\usepackage{url}
\usepackage{stmaryrd}
\usepackage{comment}
\usepackage[retainorgcmds]{IEEEtrantools}
\usepackage{hyperref}
\usepackage{geometry}
\usepackage{tikz-cd}

\title{Model theory of derivations of the Frobenius map revisited}
\author[J. GOGOLOK]{Jakub Gogolok}
\thanks{Supported by the Narodowe Centrum Nauki grant no. 2018/31/B/ST1/00357.}
\address{Instytut Matematyczny\\
Uniwersytet Wroc{\l}awski\\
Wroc{\l}aw\\
Poland}
\email{jakub.gogolok@math.uni.wroc.pl} \urladdr{http://www.math.uni.wroc.pl/\textasciitilde gogolok/}

 \DeclareMathOperator{\fr}{Fr}

 \DeclareMathOperator{\alg}{alg}

\DeclareMathOperator{\rat}{rat}

\DeclareMathOperator{\dcf}{DCF}

\newtheorem{theorem}{Theorem}[section]
\newtheorem{prop}[theorem]{Proposition}
\newtheorem{lemma}[theorem]{Lemma}
\newtheorem{cor}[theorem]{Corollary}

\theoremstyle{definition}

\newtheorem{example}[theorem]{Example}
\newtheorem{remark}[theorem]{Remark}
\newtheorem{question}[theorem]{Question}

\begin{document}
\newcommand{\lili}{\underleftarrow{\lim }}
\newcommand{\coco}{\underrightarrow{\lim }}
\newcommand{\twoc}[3]{ {#1} \choose {{#2}|{#3}}}
\newcommand{\thrc}[4]{ {#1} \choose {{#2}|{#3}|{#4}}}
\newcommand{\Zz}{{\mathds{Z}}}
\newcommand{\Ff}{{\mathds{F}}}
\newcommand{\Cc}{{\mathds{C}}}
\newcommand{\Rr}{{\mathds{R}}}
\newcommand{\Nn}{{\mathds{N}}}
\newcommand{\Qq}{{\mathds{Q}}}
\newcommand{\Kk}{{\mathds{K}}}
\newcommand{\Pp}{{\mathds{P}}}
\newcommand{\ddd}{\mathrm{d}}
\newcommand{\Aa}{\mathds{A}}
\newcommand{\dlog}{\mathrm{ld}}
\newcommand{\ga}{\mathbb{G}_{\rm{a}}}
\newcommand{\gm}{\mathbb{G}_{\rm{m}}}
\newcommand{\gaf}{\widehat{\mathbb{G}}_{\rm{a}}}
\newcommand{\gmf}{\widehat{\mathbb{G}}_{\rm{m}}}
\newcommand{\ka}{{\bf k}}
\newcommand{\ot}{\otimes}
\newcommand{\si}{\mbox{$\sigma$}}
\newcommand{\ks}{\mbox{$({\bf k},\sigma)$}}
\newcommand{\kg}{\mbox{${\bf k}[G]$}}
\newcommand{\ksg}{\mbox{$({\bf k}[G],\sigma)$}}
\newcommand{\ksgs}{\mbox{${\bf k}[G,\sigma_G]$}}
\newcommand{\cks}{\mbox{$\mathrm{Mod}_{({A},\sigma_A)}$}}
\newcommand{\ckg}{\mbox{$\mathrm{Mod}_{{\bf k}[G]}$}}
\newcommand{\cksg}{\mbox{$\mathrm{Mod}_{({A}[G],\sigma_A)}$}}
\newcommand{\cksgs}{\mbox{$\mathrm{Mod}_{({A}[G],\sigma_G)}$}}
\newcommand{\crats}{\mbox{$\mathrm{Mod}^{\rat}_{(\mathbf{G},\sigma_{\mathbf{G}})}$}}
\newcommand{\crat}{\mbox{$\mathrm{Mod}^{\rat}_{\mathbf{G}}$}}
\newcommand{\cratinv}{\mbox{$\mathrm{Mod}^{\rat}_{\mathbb{G}}$}}
\newcommand{\ra}{\longrightarrow}
\newcommand{\bdcf}{\mathcal{B}-\dcf}

\def\Ind#1#2{#1\setbox0=\hbox{$#1x$}\kern\wd0\hbox to 0pt{\hss$#1\mid$\hss}
\lower.9\ht0\hbox to 0pt{\hss$#1\smile$\hss}\kern\wd0}

\def\ind{\mathop{\mathpalette\Ind{}}}

\begin{abstract}
We prove some results about the model theory of fields with a derivation of the Frobenius map, especially that the model companion of this theory is axiomatizable by axioms used by Wood in the case of the theory $\dcf_p$ and that it eliminates quantifiers after adding the inverse of the Frobenius map to the language. This strengthens the results from \cite{DerFro}. As a by-product, we get a new geometric axiomatization of this model companion. Along the way we also prove a quantifier elimination result, which holds in a much more general context and we suggest a way of giving ``one-dimensional'' axiomatizations for model companions of some theories of fields with operators.
\end{abstract}

\maketitle

\section{Introduction}

In this paper we investigate model-theoretically the so-called \textit{derivations of the Frobenius map} building on results by Kowalski in \cite{DerFro}. Let $K$ be a field of characteristic $p>0$ and let $n$ be a natural number. A \textit{derivation of the $n$-th power of Frobenius map} on $K$ is an additive map $\partial\colon K\to K$ satisfying the following twisted Leibniz rule
$$\partial\left( xy \right) = x^{p^n}\partial\left( y \right)+y^{p^n}\partial\left( x\right)$$
for $x,y\in K$. In \cite{DerFro} they are also called $n$-derivations for short, but we choose a more expressive name \textit{$\fr^n$-derivation}. If $D\colon K\to K$ is a derivation in the usual sense, then $\fr^n\circ D$ is a $\fr^n$-derivation, but most $\fr^n$-derivations do not come from usual derivations.

Derivations of the Frobenius map have some reasonable model theory, which was explored in \cite{DerFro}. In particular, the theory of fields with a $\fr^n$-derivation (or \textit{$\fr^n$-differential fields} for short) has a model companion, this model companion is strictly stable and eliminates quantifiers in a natural language (details below). In this paper we establish some further algebraic and model-theoretic properties of $\fr^n$-derivations and answer the six questions stated in \cite{DerFro}.

We now recall some definitions and facts from \cite{DerFro}. Fix a $\fr^n$-differential field $\left( K, \partial \right)$ of characteristic $p>0$. The kernel of $\partial$, called the \textit{constants} of $\left( K, \partial \right)$, is denoted by $K^\partial$ and it is a field extension of $K^p$. We call $\left( K, \partial \right)$ \textit{strict} if $K^\partial = K^p$. There is an obvious notion of $\fr^n$-differential extensions and any $\fr^n$-differential field has a strict $\fr^n$-differential extension (Lemma 1.9 in \cite{DerFro}). Moreover, derivations of the Frobenius map extend to separable field extensions and for algebraic extensions this extension is unique. We call a $\fr^n$-differential field \textit{differentially perfect}, if any of its $\fr^n$-differential extension is separable. Since there is always a strict extension, we have that any differentially perfect field is strict. Lemma 2.1 says that the converse also holds.

An important property of derivations of the Frobenius, which we will use a few times, is the following: for $m, n>0$, the composition of an $\fr^m$-derivation and an $\fr^n$-derivation is an $\fr^{m+n}$-derivation. This does not hold for usual derivations, as compositions of derivations produces ``operators of higher order'' (a Hasse-Schmidt derivation).

Let $L^\partial$ be the language of fields together with one unary function symbol $\partial$. Clearly, there is an $L^\partial$-theory $\fr^n-\operatorname{DF}_p$ (called $\operatorname{DF}_{p,n}$ in \cite{DerFro}), whose models are precisely $\fr^n$-differential field of characteristic $p$. We can and will also consider this theory in the languages $L^\partial_\lambda$ and $L^\partial_{\lambda_0}$, where we add the function symbols for the $\lambda$-functions and only $\lambda_0$ respectively (see \cite[Section 1.8]{ChGeneric} for the definition of $\lambda$-functions). Recall that $\lambda_0$ is interpreted in a field $K$ as a function, which is the inverse of the Frobenius on $K^p$ and zero everywhere else. The theory $\fr^n-\operatorname{DF}_p$ has a model companion in the language $L^\partial$, which we call $\fr^n-\dcf_p$. The axioms of this theory are geometric (see the beginning of Section 2 in \cite{DerFro}), of a similar form as the geometric axioms for $\dcf_0$ (see \cite{PiercePillay}). Unfortunately, there is an unnaturally looking assumption in the geometric axioms of $\fr^n-\dcf_p$ about the density of a certain equalizer. The results of this paper provide a more elegant geometric axiomatization of $\fr^n-\dcf_p$. A possibility for such an axiomatization was also the content of Question 4 in \cite{DerFro}. Since we will only use the fact that $\fr^n-\operatorname{DF}_p$ has a model companion and not the specific axiomatization of it, we will not recall the original geometric axioms and instead just present the nicer version following from our work (see Remark \ref{geo}).

The theory $\fr^n-\dcf_p$ eliminates quantifiers in the language $L^\partial_\lambda$ (see Theorem 2.2(v) in \cite{DerFro}). In \cite{DerFro} it is wrongly claimed that $\fr^n-\dcf_p$ is also a model companion of $\fr^n-\operatorname{DF}_p$ in the language $L^\partial_\lambda$, but this can be easily fixed (see the discussion above Remark \ref{redu}). In Section 2. we prove that $\fr^n-\dcf_p$ eliminates quantifiers already in the language $L^\partial_{\lambda_0}$, answering Question 3 from \cite{DerFro}. Actually, in Lemma \ref{lindisj} and Lemma \ref{lindisj2} we provide an easy proof of a surprisingly general fact, which implies quantifier elimination results for a very general class of theories of fields with operators. More precisely (see Subsection \ref{general}), we give a new and shorter proof of quantifier elimination in the framework of $B$-operators (introduced in \cite{BHKK}) and a stronger quantifier elimination result in the setting of $\mathcal{B}$-operators (introduced in \cite{beta-op}).

In Section 3. we prove that $\fr^n-\dcf_p$ can be axiomatized by the same axioms (Theorem \ref{main}), which were used by Wood to axiomatize $\dcf_p$ (see \cite{Wo1}). This result answers Question 4 and Question 5 from \cite{DerFro}. To formulate Wood's axioms we need the concept of differential polynomials, but this works just as in the ($\fr^0$-)differential case. Namely, for a $\fr^n$-differential field $\left( K, \partial \right)$ we consider the polynomial ring $K\left\{ X \right\} = K \left[ X, X',\dotsc , X^{\left( i \right)}, \dotsc \right]$ in countably many variables, and consider it with the unique $\fr^n$-derivation $\partial'$ on $K\left\{ X \right\}$, which extends $\partial\colon K \to K$ and has the property that $\partial\left( X^{\left( i \right)}\right) = X^{\left( i+1 \right)}$. By the usual abuse of notation, we write $\partial$ instead of $\partial'$. We call elements of $K\left\{ X \right\}$ differential polynomials with coefficients in $K$. For $f\in K\left\{ X \right\}$ and $a\in K$, we may evaluate $f$ at $a$ in an obvious manner, an we write $f\left( a \right)$ for this evaluation.. The order of $f\in K\left\{ X \right\}\setminus \left\{ 0 \right\}$ is the biggest $i$ such that $X^{\left( i \right)}$ appears in $f$, if $f$ is not constant, and $-1$ if $f$ is a non-zero constant. Let us denote by $T_{\operatorname{Wood}}$ the following scheme of axioms: a $\fr^n$-differential field $\left( K,\partial \right)$ satisfies $T$ if and only if
\begin{enumerate}\item $K$ is strict
    \item for any non-zero differential polynomials $f, g\in K\left\{ X \right\}$, where the order of $f$ is equal $m$ and the order of $g$ is smaller than $m$, if $\frac{\partial f}{\partial X^{\left( m \right)}}\neq 0$, then there is some $a\in K$ such that $f\left( a \right) = 0$ and $g\left( a \right)\neq 0$.
\end{enumerate}
It is clear that the axioms above are expressible in $L^\partial$. Moreover, the sentences expressing these axioms are verbatim the same as the Wood axioms for $\dcf_p$ in \cite{Wo1} (i. e. they do not depend on $n$).

The proofs in Section 3. are in the spirit of the proof from \cite{Wo1} that the Wood axioms work for $\dcf_p$, although we had to overcome some obstacles. For example, for usual derivations, the ring of differential polynomial has nice division properties, which is not the case for $\fr^n$-derivations for $n>0$, since e. g. for $a\in K$ we have 
$$\partial\left( aX \right) = a^{p^n}X' +\partial \left( a \right) X^{p^n},$$
so applying $\partial$ increases the degree in lower-order variables, which complicates possible division algorithms. Because of this lack of division, we can not use ``constrained ideals'' as in \cite{Wo1}, but we still can reason using some ``choosing polynomials of minimal order/degree'' type of reasoning, which is somewhat reminiscent of a division algorithm.

\vspace{0.5cm}

In what follows $n$ is always a positive natural number (i. e. we do not speak about usual derivations, when considering $\fr^n$-derivations), all fields are of characteristic $p>0$ and we set $q=p^n$, so that an $\fr^n$-differential field $\left( K, \partial \right)$ obeys
$$\partial\left( xy \right) = x^{q}\partial\left( y \right)+y^{q}\partial\left( x\right)$$
for all $x,y\in K$.

\section{Quantifier elimination}
Let $\left( K, \partial \right)$ be a field of characteristic $p$ with some \textit{operators}, i. e. a tuple (possibly infinite) of unary functions $\partial = \left( \partial_i\colon K\to K \right)_{i \in I}$. We define the \textit{constants} of $\left( K, \partial\right)$ as the set of common zeroes of all $\partial_i$ and denote it by $K^\partial$. We assume that $K^\partial$ is a field. We also assume that every $\partial_i$ is additive and ``$\fr^n$-linear over the constants'', i. e. there is some natural number $m_i$ such that for any $a\in K^\partial,\ x \in K$ we have $\partial\left( ax \right) = a^{p^{m_i}} \partial\left( x \right)$. Examples of such operators include derivations, difference operators, Hasse-Schmidt derivations and derivations of the Frobenius map (see also Subsection \ref{general}).
 
Let $\left( K, \partial \right) \subseteq \left( L, \partial' \right)$ be an extension of fields with operators in the above sense, that is, $\partial'|_K = \partial$ and the numbers $m_i$ mentioned above are the same for $K$ and $L$. By abuse of notation, we will use the same symbol $\partial$ for the operators on $K$ and on $L$.

\begin{lemma}
\label{lindisj}
Let $\left( K, \partial \right) \subseteq \left( L, \partial \right)$ be as above. Assume that $K$ is strict, i. e. $K^\partial = K^p$. Then $L^\partial$ and $K$ are linearly disjoint over $K^\partial$.
\end{lemma}
\begin{proof}
Assume the conclusion is not true and take the minimal $n>1$ such that there are some $x_1,\dotsc, x_n\in L^\partial$ are linearly dependent over $K$, but linearly independent over $K^\partial$. By the minimality assumption, there are $a_1, \dotsc , a_n\in K\setminus \left\{ 0 \right\}$ such that:

$$a_1x_1+\dotsc + a_nx_n = 0$$
Then for any $i\in I$:

$$0=\partial_i\left( \frac{a_1}{a_n}x_1+\dotsc + \frac{a_{n-1}}{a_n}x_{n-1} +x_n\right)=
\partial_i\left(\frac{a_1}{a_n}\right)x_1^{p^{m_i}}+\dotsc +  \partial_i\left(\frac{a_{n-1}}{a_n}\right)x_{n-1}^{p^{m_i}}$$

If some $\partial_i \left( \frac{a_j}{a_n} \right)$ is nonzero, then $x_1^{p^{m_i}},\dotsc , x_{n-1}^{p^{m_i}}$ are linearly dependent over $K$, so by the minimality assumption on $m$ we get that $x_1^{p^{m_i}},\dotsc , x_{n-1}^{p^{m_i}}$ are linearly dependent over $K^\partial=K^p$, hence $x_1^{p^{m_i-1}}, \dotsc , x_{n-1}^{p^{m_i-1}}$ are linearly dependent over $K$. Repeating this reasoning yields that $x_1,\dotsc, x_{n-1}$ are linearly dependent over $K^p$, contrary to the assumption, that they independent over $K^\partial=K^p$.

Therefore for any $i$ we have $\partial_i\left( \frac{a_1}{a_n} \right) = \dotsc = \partial_i\left(\frac{a_{n-1}}{a_n}\right)=0$, hence $\frac{a_1}{a_n},\dotsc , \frac{a_{n-1}}{a_n}\in K^\partial$. By the strictness assumption we get that for some $b_1,\dotsc, b_{n-1}\in K\setminus \left\{ 0 \right\}:$
$$a_1=b_1^p a_n, \dotsc , a_{n-1}=b_{n-1}^p a_n,$$
thus
$$0 = a_1x_1+\dotsc + a_nx_n = a_n \left( b_1^p x_1 + \dotsc +b_{n-1}^p x_{n-1} + x_n \right),$$
hence $x_1, \dotsc , x_n$ are linearly dependent over $K^p=K^\partial$, contrary to the assumption.
\end{proof}

From the proof it is clear that in the linear case (i. e. for every $i$ we have $m_i=0$) we have the following

\begin{lemma}
\label{lindisj2}
Let $\left( K, \partial \right) \subseteq \left( L, \partial \right)$ be an extension of fields with operators linear over the constants. Then $L^\partial$ and $K$ are linearly disjoint over $K^\partial$.
\end{lemma}

The strictness assumption in Lemma \ref{lindisj} is necessary (which gives a negative answer to Question 1 in \cite{DerFro}), as shown by the example below.

\begin{example}
Take $K=\mathbb{F}_p\left( X, Y, \lambda, \mu \right), L=\mathbb{F}_p\left( X^{1/p}, Y^{1/p}, \lambda, \mu \right)$ and define a derivation of the Frobenius map on $L$ by setting
$$\partial \left( X^{1/p} \right) = \partial \left( Y^{1/p} \right) = 0, \hspace{0.5cm} \partial \left( \lambda \right) = Y, \hspace{0.5cm} \partial\left( \mu \right) = -X.$$
We will show that $L^\partial$ and $K$ are not linearly disjoint over $K^\partial$. Note that 
$$\partial \left( \lambda X^{1/p} + \mu Y^{1/p} \right) = X\partial \left( \lambda \right)+ Y \partial \left( \mu \right) = 0.$$
Thus, $X^{1/p}, Y^{1/p}, \lambda X^{1/p} + \mu Y^{1/p}$ are elements of $L^\partial$, linearly dependent over $K$. However, they are independent over $K^\partial$: 

Indeed, for any $a, b, c \in K^\partial$, if
$$aX^{1/p}+bY^{1/p} + c\left( \lambda X^{1/p} + \mu Y^{1/p} \right)=0$$
then 
$$\left( a+c\lambda \right) X^{1/p} + \left( b+c\mu \right) Y^{1/p} = 0,$$
but $X^{1/p}$ and $Y^{1/p}$ are linearly independent over $K$, so $a+c\lambda = b+c\mu = 0.$ If $c\neq 0$, then $\lambda = -\frac{a}{c}\in K^\partial$, which is not the case. Thus $c=0$ and therefore $a=b=0$, hence $X^{1/p}, Y^{1/p}, \lambda X^{1/p} + \mu Y^{1/p}$ are linearly independent over $K^\partial$.
\end{example}

Lemma \ref{lindisj} answers Question 2 from \cite{DerFro}, which asks whether strict $\fr^n$-differential fields are differentially perfect. As explained there, it also implies a positive answer to Question 3, i. e. 

\begin{theorem}
\label{qe}
The theory $\fr^{n}-\dcf_p$ eliminates quantifiers in the language $L_{\lambda_0}^{\partial}$.
\end{theorem}

\subsection{Quantifier elimination for more general operators.}
\label{general}
Lemma \ref{lindisj} and Lemma \ref{lindisj2} are very general and actually provide quantifier elimination in more general frameworks, which we will now briefly discuss. 

We begin with the setting of $B$-operators, considered in \cite{BHKK}. Let $k$ be a field, let $B$ be a finite $k$-algebra together with a $k$-algebra homomorphism $\pi\colon B\to k$. A \textit{$B$-operator} on a $k$-algebra $R$ is then a $k$-algebra homomorphism $\partial\colon R\to R\otimes_k B$ such that the following diagram commutes
\begin{center}
  \begin{tikzcd}
R \arrow[r, "\partial"] \arrow[rd, "\operatorname{id}_R"'] & R\otimes_k B \arrow[d, "\operatorname{id}_R\otimes\pi"] \\
                                                           & R   
\end{tikzcd}
\end{center}
Fix a basis $b_0,\dotsc , b_d$ of $B$ over $k$ such that $\pi\left( b_0 \right) = 1$ and $\pi\left( b_i \right) = 0$ for $i>0$. a $B$-operator is then the same as a $d$-tuple of maps $\partial_1,\dotsc , \partial_d\colon R\to R$ such that the map
$$R\ni r \mapsto r\otimes b_0 + \partial_1\left( r \right)\otimes b_1+\dotsc+ \partial_d\left( r \right)\otimes b_d\in R\otimes_k B$$
is a homomorphism of $k$-algebras.

\begin{example}
\label{auto}
If $B=k\times k$ and $b_0=\left( 1, 0 \right), b_1=\left( 0, 1 \right)$, then $\partial_1\colon R\to R$ is a $B$-operator if and only if $\partial_1$ is and $k$-algebra endomorphism.
\end{example}

The \textit{ring of constants} of a $B$-operator $\partial$ on $R$ is defined as
$$R^\partial=\left\{  r\in R\colon \partial\left( r\right) = r\otimes 1_B\right\}.$$
Note that in general this is not the same as the constants in our sense, i. e. the intersection of the kernels of $\partial_1,\dotsc , \partial_d$. For example, consider a $k$-algebra automorphism $\partial_1$ of $R$, which is a $B$-operator as in Example \ref{auto}. Then the ring of constants in our sense is $\left\{ 0 \right\}$, but the ring of constants defined above is the fixed field of $\partial_1$. 

However, if we take the basis $b_0,\dotsc , b_d$ of $B$ over $k$ so that $b_0=1$, then both notions of constants agree. Moreover, it is then easy to calculate that such operators are linear over the constants, so they fit into our considerations. Therefore, we may apply Lemma \ref{lindisj2}, which is in this instance the same as Corollary 4.5 in \cite{BHKK}, namely

\begin{cor}
\label{lindisjB}
Let $\left( K, \partial \right) \subseteq \left( L, \partial \right)$ be an extension of $B$-fields (i. e. fields with $B$-operators). Then $L^\partial$ and $K$ are linearly disjoint over $K^\partial$.
\end{cor}

Our proof is easier than the proof in \cite{BHKK}. This corollary is then applied as follows. The main result of \cite{BHKK} is a complete classification for which $B$ the theory of $B$-fields has a model companion $B-\dcf$ (see Corollary 3.9 there for details). This happens in essentially two cases, one of which is when $B$ is local and the nilradical of $B$ coincides with the kernel of the Frobenius homomorphisms on $B$ (e. g. the case of derivations). Corollary \ref{lindisjB} is used to prove the following (Theorem 4.11 in \cite{BHKK}):

\begin{theorem}
\label{qeB}
Assume that $B$ is local and the nilradical of $B$ coincides with the kernel of the Frobenius homomorphisms on $B$. Then the theory $B-\dcf$ has quantifier eliminations in the language $L_{\lambda_0}^\partial$.
\end{theorem}

Even more generally, we can get a quantifier-elimination result for so-called $\mathcal{B}$-operators considered in \cite{beta-op}. We will not recall the precise definition of a $\mathcal{B}$-operator, as it is a bit involved, but the rough idea is to replace the functor $-\otimes_k B$ in the context of $B$-operators by an appropriate functor $\mathcal{B}$ and the map $\pi$ by a natural transformation $\mathcal{B}\longrightarrow \operatorname{id}$.

Theorem 2.19 in \cite{beta-op} says, that we can replace $\mathcal{B}$ by an isomorphic $\mathcal{B}'$ in a somewhat ``normal form''. This normal form has the property that there is some $k$-algebra $B$ such that $B$-operators ``twisted'' by an appropriate sequence of Frobenius maps are $\mathcal{B}$-operators. In other words, Theorem 2.19 in \cite{beta-op} says that general $\mathcal{B}$-operators are related to $B$-operators in the same way as derivations of the Frobenius map are related to derivations. Now, in order to speak about $\mathcal{B}$-operators as tuples of maps $\partial_1, \dotsc , \partial_d \colon R\to R$ (as opposed to a $k$-algebra homomorphism $\partial\colon R\to \mathcal{B}\left( R \right)$) it is enough to fix a basis $b_0,\dotsc , b_d$ of this algebra $B$ over $k$. If we choose it so that $b_0=1$, then tuple $\partial_1,\dotsc , \partial_d$ is ``$\fr^n$-linear over the constants'' (see the proof of Lemma 4.16 in \cite{beta-op} for details), so they fall under the class of operators considered in Lemma 2.1. 

One of the main results of \cite{beta-op} says that the theory of $\mathcal{B}$-fields has a model companion $\bdcf$, provided that the theory of $B$-field is companionable, where $B:=\mathcal{B}\left( k^{\alg}\right)$. If $B$ satisfies the assumptions of Theorem \ref{qeB}, then we can use our Lemma \ref{lindisj} to show the following
\begin{theorem}
\label{qeBeta}
Assume that $B:=\mathcal{B}\left( k^{\alg}\right)$ satisfies the assumptions of Theorem \ref{qeB}. Then, the theory $\mathcal{B}-\dcf$ eliminates quantifiers in the language $L^\partial_{\lambda_0}$.
\end{theorem}
The proof is completely analogous to the proofs of Theorem \ref{qe} and Theorem \ref{qeB}. Theorem \ref{qeBeta} is new and strengthens Theorem 4.14 from \cite{beta-op}, which states that $\mathcal{B}-\dcf$ eliminates quantifiers in the language $L_{\lambda}^\partial$. This theorem is also a vast generalization of Theorem \ref{qe}.

\section{Wood axioms for \texorpdfstring{$\operatorname{Fr}^n-\operatorname{DCF}_p$}{Fr\^{} n-DCF\_ p}}

We are now going to prove, that the Wood axioms $T$ described in the Introduction do axiomatize $\fr^n-\dcf_p$. The proof consists of two steps. First we prove some sort of primitive element theorem for $\fr^n$-derivations (Proposition \ref{primitive}), which implies that for $\fr^n$-differential fields 1-existential closedness is equivalent to existential closedness (Corollary \ref{ec}). In the second step, we show that Wood axioms, as expected, axiomatize 1-existentially closed $\fr^n$-differential fields (Theorem \ref{main}).

The proof in \cite{Wo1} that $\dcf_p$ can be axiomatized by the Wood axioms relays on the following primitive element theorem (Theorem 1 in \cite{Seidenberg}):
\begin{theorem}
\label{seid}
Let $\left( K, \partial \right)$ be a differentially perfect differential field which has infinite dimension over its constants (equivalenty: every non-zero differential polynomial assumes a non-zero value). Let $K\subseteq L$ be an extension of differential fields and let $a_1,\dotsc , a_n\in L$ be differential-algebraic over $K$. Then there is some $c\in L$ such that $K\left\langle a_1, \dotsc , a_n \right\rangle = K\left\langle c \right\rangle$. Moreover, $c$ can be taken from $Ka_1+\dotsc + K a_n$.
\end{theorem}

It seems that the above fact does not hold for $\fr^n$-derivations, but in any case, we have the following weaker version, which suffices for our purposes. The proof is similar to the proof of Theorem \ref{seid} given in \cite{Seidenberg}, although additional steps are needed.

\begin{prop}
\label{primitive}
Let $n>0$ and let $\left( K, \partial \right)$ be a differentially perfect $\fr^n$-differential field such that that every non-zero differential polynomial assume a non-zero value. Let $K\subseteq L$ be an extension of $\fr^n$-differential fields and let $a_1,\dotsc , a_m\in L$ be differential-algebraic over $K$. Then there is some $c\in L$ and some $i>0$ such that $a_1^{q^i}, \dotsc , a_m^{q^i} \in K\left\langle c \right\rangle$. Moreover, $c$ can be taken from $Ka_1+\dotsc + K a_n$.
\end{prop}
\begin{proof}
It is enough to prove the theorem for $m=2$. Let $u, v\in L$ be differentially algebraic over $K$. Let $\Lambda$ be a differentially transcendetal indeterminate. The elements $u, v, \Lambda\in L\left\langle \Lambda \right\rangle$ are differentially algebraic over $K\left\langle \Lambda \right\rangle$, hence so is $u+\Lambda v$, thus there is some non-zero polynomial $G\in K\left[ X_0, \dotsc, X_t, Y_0, \dotsc , Y_s\right]$ such that
\[
G\left( \Lambda, \dotsc, \Lambda^{\left( t \right)},u+\Lambda v, \dotsc , \left( u+\Lambda v\right)^{\left( s \right)} \right)= 0. \tag{$\ast$}
\]
Without out loss of generality assume that the total degree of $G$ is minimal.

Fix $i\in\left\{ 0, \dotsc , s \right\}$. To simplify the notation, we set
$$\bar{\Lambda} = \left( \Lambda, \dotsc, \Lambda^{\left( t \right)},u+\Lambda v, \dotsc , \left( u+\Lambda v\right)^{\left( s \right)} \right).$$
Let $\frac{\partial}{\partial \Lambda^{\left( i \right)}}$ denote the obvious partial derivation on the field $L \left\langle\Lambda \right\rangle$. Using the chain rule, we may differentiate the equality $\left( \ast \right)$ with respect to $\Lambda^{\left( i \right)}$
\begin{align*}
0&=\frac{\partial}{\partial \Lambda^{\left( i \right)}}G\left( \Lambda, \dotsc, \Lambda^{\left( t \right)},u+\Lambda v, \dotsc , \left( u+\Lambda v\right)^{\left( s \right)} \right) \\ 
&=\sum_{k=0}^t \frac{\partial G}{\partial X_k}\left( \bar{\Lambda} \right) \cdot \frac{\partial \Lambda^{\left( k \right)}}{\partial \Lambda^{\left( i \right)}} + \sum_{k=0}^s \frac{\partial G}{\partial Y_k}\left( \bar{\Lambda} \right) \cdot \frac{\partial \left( u+\Lambda v\right)^{\left( k \right)} }{\partial \Lambda^{\left( i \right)}} \\
&= \frac{\partial G}{\partial X_i}\left( \bar{\Lambda} \right) + \frac{\partial G}{\partial Y_i}\left( \bar{\Lambda} \right) \cdot v^{q^i},
\end{align*}
since for $i>0$ we have
$$\frac{\partial}{\partial \Lambda^{\left( i \right)} }\left( u+\Lambda v\right)^{\left( k \right)} = \frac{\partial}{\partial \Lambda^{\left( i \right)} }\left( u^{\left( k \right)}+ \Lambda^{\left( k \right)}v^{q^k} + \Lambda^{q^k}v^{\left( k \right)} \right) = \delta_{ik} v^{q^k}$$
and for $i=0$
$$\frac{\partial}{\partial \Lambda}\left( u+\Lambda v\right)^{\left( k \right)} = \frac{\partial}{\partial  \Lambda}\left( u^{\left( k \right)}+ \Lambda^{\left( k \right)}v^{q^k} + \Lambda^{q^k}v^{\left( k \right)} \right) = \delta_{0k} v^{q^k}+ q^k\Lambda^{q^k-1}v^{\left( k \right)}=\delta_{0k} v^{q^k},$$
where $\delta_{ik}$ is the Kronecker delta.
\vspace{0.5cm}

\textbf{Claim 1.} There is some $i\in\left\{ 0, \dotsc , s \right\}$ such that $\frac{\partial G}{\partial Y_i}$ is not the zero polynomial.

\textit{Proof of Claim 1.: } Assume this is not the case. Then also each $\frac{\partial G}{\partial X_i}$ is the zero polynomial - otherwise, $\frac{\partial G}{\partial X_i}\left( \bar{\Lambda} \right) = 0$ by the formula above, which would contradict the minimality assumptions on $G$. Thus $G$ is a polynomial in $X_0^p,\dotsc , X_t^p, Y_0^p, \dotsc Y_s^p$. Therefore $( \ast )$ expresses the linear dependence over $K$ of $p$th powers of some monomials in $\Lambda, \dotsc, \Lambda^{\left( t \right)},u+\Lambda v, \dotsc , \left( u+\Lambda v\right)^{\left( s \right)}$. Lemma \ref{lindisj} implies that the $p$th powers of these monomials are already linearly dependent over $K^p$, hence the monomials are linearly dependent over $K$. Since $\Lambda, \dotsc, \Lambda^{\left( t \right)}$ are algebraically independent over $K$, some $\left( u+\Lambda v\right)^{\left( i\right)}$ must appear in this dependence relation, thus we get a relation of the form
$$F\left( \Lambda, \dotsc, \Lambda^{\left( t \right)},u+\Lambda v, \dotsc , \left( u+\Lambda v\right)^{\left( s \right)} \right)= 0,$$
where $F$ is a multivariate polynomial over $K$ with lower total degree than $G$, contrary to the minimality assumption. \qed
\vspace{0.5cm}

Thus, for some $i\in\left\{ 0, \dotsc , s \right\}$ we have the following equality:
$$v^{q^i} = - \frac{\frac{\partial G}{\partial X_i} \left( \bar{\Lambda} \right) }{ \frac{\partial G}{\partial Y_i}\left( \bar{\Lambda} \right) }.$$

\textbf{Claim 2.} There is some $\lambda\in K$ such that 
$$\frac{\partial G}{\partial Y_i} \left( \lambda, \dotsc, \partial^t\left( \lambda \right),u+\lambda v, \dotsc , \partial^s\left( u+\lambda v\right)\right)\neq 0.$$

\textit{Proof of Claim 2.: } Here $\frac{\partial G}{\partial Y_i}$ is a non-zero differential polynomial, but not over $K$, so we cannot use the assumption about $K$ directly. However, expanding $u, v$ according to some transcendence basis of $L$ over $K$, we may rewrite the above (desired) inequality as a conjunction of polynomial inequalities over $K$ of the form
$$\bigwedge_{k=0}^N f_k\left( \lambda \right) \neq 0$$
for some $f_0,\dotsc , f_N\in K\left\{ X \right\}$. But this conjunction is equivalent to the single inequality $\prod_{k=0}^N f_k\left( \lambda \right) \neq 0$, to which the assumption about $K$ applies, yielding the desired $\lambda$.
\qed
\vspace{0.5cm}

For this $\lambda$ one has therefore $v^{q^i}\in K\left\langle u+\lambda v \right\rangle$ and it follows that $u^{q^i}\in K\left\langle u+\lambda v \right\rangle$, which finishes the proof.
\end{proof}

\begin{remark}
Theorem 2.2(v) in \cite{DerFro} states (in our notation) that $\fr^n-\dcf_p$ is a model companion of $\fr^n-\operatorname{DF}$ in $L^\partial_\lambda$. This is not true however, as proved by the following example. Let $K$ be a model of $\fr^n-\operatorname{DF}$ which is not strict (e. g. any imperfect field with the zero $\fr^n$-derivation) and $K\subseteq L$ be an $L^\partial$-extension of $K$ such that $L$ is a model of $\fr^n-\dcf_p$. Since $L$ is strict and $K$ is not, this extension is not separable (as $L$ contains $p$-roots not present in $K$), thus $K\subseteq L$ is not an $L^\partial_{\lambda}$-extension. Hence $\fr^n-\dcf_p$ is not a model companion of $\fr^n-\operatorname{DF}$ in $L^\partial_\lambda$.

It can be easily fixed by replacing $\fr^n-\operatorname{DF}$ by $\fr^n-\operatorname{DF} + \mbox{ ``strictness'' }$ and the proof from \cite{DerFro} goes through.
\end{remark}

For the above reasons, in the following few auxiliary facts about existential closedness, the strictness assumption is present.

\begin{lemma}
\label{redu}
Let $K$ be a strict $\fr^n$-differential field and let $\phi\left( x \right)$ be a quantifier-free $L^\partial_{\lambda_0}$-formula with parameters from $K$, which has a solution $a$ in some $L^\partial_{\lambda_0}$-extension $L\supseteq K$. Then there exists a quantifier free $L^\partial$-formula $\psi \left( x \right)$ with parameters from $K$, satisfiable in some $L^\partial_{\lambda_0}$-extension of $K$, which has the following property: in any $L^\partial_{\lambda_0}$-extension $L\supseteq K$ (equivalenty: $L^\partial$-extension) we have $L\models \psi \left( x \right) \rightarrow \phi\left( x \right)$.
\end{lemma}
\begin{proof}
When plugging $a$ into $\phi\left( x \right)$, any instance of $\lambda_0\left( b \right)$ becomes either zero or $b^{1/p}$. Applying a high enough power of the Frobenius automorphisms to every equality and adding some formulas of the form $\partial \left( t\left( x \right) \right) =0$ for some $L^\partial$-terms $t$ (this formulas keep track whether $a\models\lambda_0\left( t\left( x \right) \right) = 0$), we get a quantifier free $L^\partial$-formula $\psi \left( x \right)$ with the following property: in any $L^\partial_{\lambda_0}$-extension $L\supseteq K$ (equivalenty: $L^\partial$-extension) we have $L\models \psi \left( x \right) \rightarrow \phi\left( x \right)$, as desired.
\end{proof}

\begin{prop}
\label{1ec}
In the language $L^\partial_{\lambda_0}$, strict 1-existentially closed $\fr^n$-differential fields are existentially closed.
\end{prop}
\begin{proof}
Let $K$ be a strict $\fr^n$-differential field, which is 1-existentially closed  in $L^\partial_{\lambda_0}$. Let $\phi$ be an existential $L^\partial_{\lambda_0}$-sentence with parameters from $K$ and let $K\subseteq L$ be an extension of $\fr^n$-differential fields such that $L$ satisfies $\phi$. By Lemma \ref{redu} we may assume that $\phi$ is an $L^\partial$-sentence and by the standard tricks we may assume that $\phi$ is of the form $\left( \exists x \right) \left( \bigwedge_{i=1}^k f_i\left( x \right) = 0 \right)$ where $x$ is a tuple of variables and $f_1,\dotsc , f_k$ are multivariate differential polynomials with coefficients from $K$. 

Take a tuple $a = \left( a_1,\dotsc a_m \right)$ from $L$ such that $f_1\left( a \right) = \dotsc = f_k\left( a \right) = 0$. We may assume that every element of $a$ is differentially algebraic over $K$.\footnote{No finite set of formulas in variable $x$ can imply the condition ``$x$ is differentially transcendental over $K$'', thus the set of formulas $\left\{ f_1\left( x \right) = \dotsc = f_k\left( x \right) = 0\right\}$ extends to a complete type over $K$, whose realization over $K$ is differentially algebraic over $K$.} Then, by Lemma \ref{primitive} there is some single element $c\in L$ such that for each $i$ we have $a_i^{q^N}\in K\left\langle c \right\rangle$ for some natural number $N$. Thus there are univariate differential polynomials $g_1, h_1, \dotsc , g_m, h_m$ over $K$, such that 
$$a_1^{q^N}=\frac{g_1\left( c \right)}{h_1\left( c \right)}, \dotsc , a_m^{q^N}=\frac{g_m\left( c \right)}{h_m\left( c \right)}$$
or in other words 
$$a_1=\lambda_0^{nN}\left( \frac{g_1\left( c \right)}{h_1\left( c \right)} \right) , \dotsc , a_m=\lambda_0^{nN}\left( \frac{g_m\left( c \right)}{h_m\left( c \right)} \right),$$
since $q=p^n$. Consider the following formula:
$$\psi\left( y \right) \colon  \ \bigwedge_{i=1}^k   h_i\left( y \right)\neq 0  \wedge\bigwedge_{i=1}^k f_i\left( \lambda_0^N\left( \frac{g_1\left( y \right)}{h_1\left( y \right)} \right), \dotsc , \lambda_0^N\left( \frac{g_m\left( y \right)}{h_m\left( y \right)} \right) \right) = 0. $$
Then $L\models \psi\left( c \right)$, therefore by $1$-existential closedness of $K$ in $L^\partial_{\lambda_0}$ there is some $b\in K$ such that $L\models \psi\left( b \right)$. Then the tuple $ \left( \lambda_0^N\left( \frac{g_1\left( b \right)}{h_1\left( b \right)} \right), \dotsc , \lambda_0^N\left( \frac{g_m\left( b \right)}{h_m\left( b \right)} \right)\right)$ witnesses the satisfiability of $\phi$ in $K$, as desired.
\end{proof}

\begin{cor}
\label{ec}
For a strict $\fr^n$-differential field $\left( K, \partial \right)$, the following conditions are equivalent:
\begin{enumerate}
    \item $K$ is a $1$-existentially closed $\fr^n$-differential field in the language $L^\partial$.
    \item $K$ is a existentially closed $\fr^n$-differential field in the language $L^\partial$.
    \item $K$ is a $1$-existentially closed $\fr^n$-differential field in the language $L^\partial_{\lambda_0}$.
    \item $K$ is a existentially closed $\fr^n$-differential field in the language $L^\partial_{\lambda_0}$.
\end{enumerate}
\end{cor}
\begin{proof}
The implications $(2)\Longrightarrow (1)$ and $(4)\Longrightarrow (3)$ are obvious. Corollary \ref{1ec} gives the implication $(3)\Longrightarrow (4)$. It is now enough to prove $(1)\Longleftrightarrow (3)$ and $(2)\Longleftrightarrow (4)$.

By Lemma 2.1, $K$ has the same extensions in the languages $L^\partial$ and $L^\partial_{\lambda_0}$, hence $(3)\Longrightarrow (1)$ and $(4)\Longrightarrow (2)$. Moreover, by this and by Lemma \ref{redu}, if $K$ is (1-)existentially closed in $L^\partial$, it is also (1-)existentially closed in $L^\partial_{\lambda_0}$, therefore $(1)\Longrightarrow (3)$ and $(2)\Longrightarrow (4)$, which finishes the proof.
\end{proof}

\begin{remark}
\label{strategy}
Corollary \ref{ec} is the key ingredient of the proof of our main result, Theorem \ref{main} (i. e. that $T_{\operatorname{Wood}}$ axiomatizes $\fr^n-\dcf$). Using this Corollary, we are basically just left to prove, that models of $T_{\operatorname{Wood}}$ are $1$-existentially closed $\fr^n$-differential fields, a statement which seems natural.

It seems that idea of the proof of Theorem \ref{main} can be applied in a more general context. Whenever we have a theory $T'$ of fields with some operators satisfying some primitive element theorem (e. g. as in \cite{pogudin}), then $1$-existential closedness is equivalent to existential closeness. On the other hand, if a theory of fields $T'$ has a model companion which eliminates quantifiers (in some reasonable language $L$), then it is easy to see that the class of $1$-existentially closed models of $T'$ in the language $L$ is elementary. A more explicit axiomatization of $1$-existentially closedness models, so ``Wood axioms for $T'$'', should come from a division algorithm for the analog of the ring of differential polynomials for $T'$. Thus, we find it appropriate to ask the following.
\end{remark}

\begin{question}
Can the above strategy be carried out in some interesting case, for example in the case of $B$-operators?
\end{question}

\begin{remark}
It is worth noting, that the idea in Remark \ref{strategy} does not apply $\operatorname{ACFA}$. Recall that $\operatorname{ACFA}$ has quantifier eliminations only up to a single existential quantifier. Moreover, the fixed field of a model of $\operatorname{ACFA}$ is pseudofinite field, a condition which is not axiomatizable by formulas in one variable, hence it seems unlikely that $\operatorname{ACFA}$ can be axiomatized so.
\end{remark}

In what follows, we will use the following notation: if $K$ is a $\fr^n$-differential field, $a\in K$ and $m>0$, then $\partial^{<m}\left( a \right)$ denotes the tuple $\left( a, \partial\left( a \right), \dotsc, \partial^{m-1}\left( a \right) \right)$ and analogously for $\partial^{\le m} \left( a \right)$.

\begin{lemma}
\label{sep}
Let $K$ be a strict $\fr^n$-differential field, $K\subseteq L$ an $\fr^n$-differential field extension and $a\in L$. For any $m>0$, if $\partial^m \left( a \right)$ is algebraic over $K\left( \partial^{<m} \left( a \right) \right)$, then $\partial^m \left( a \right)$ is separably algebraic over $K\left( \partial^{<m} \left( a \right) \right)$.
\end{lemma}
\begin{proof}
Note that if $\partial^m \left( a \right)$ is separably algebraic over $K\left( \partial^{ m} \left( a \right) \right)$, then for any $k>m$ we have $\partial^{k}\left( a \right) \in K\left( \partial^{\le m} \left( a \right) \right)$. Indeed, if $f$ is the (separable) minimal polynomial of $\partial^m \left( a \right)$ over $K\left( \partial^{<m} \left( a \right) \right)$, then
$$0=\partial \left( f \left( \partial^m \left( a \right) \right) \right) =  f^\partial \left( \partial^m \left( a \right)^q \right) + f'\left( \partial^m \left( a \right) \right)^q \partial^{m+1} \left( a \right)$$
and since $f'\left( \partial^m \left( a \right) \right)\neq 0$, the claim follows. Here $f^\partial$ denotes the polynomial obtained by applying $\partial$ to the coefficients of $f$.

We may therefore assume that the tuple $\bar{a}=\partial^{<m} \left( a \right)$ is algebraically independent over $K$. Let $f\left( X \right) = \sum_{i=0}^N f_i\left( \bar{a} \right) X^i$ be the minimal polynomial of $b=\partial^{m} \left( a \right)$ over $K\left[ \bar{a}\right]$ - by this we mean that $f\in K\left[ \bar{a}\right] \left[ X \right]$ is a minimal-degree ($=N$) polynomial vanishing on $b$ and the coefficients $f_i$ have minimal total degree. Since the tuple $\bar{a}$ is algebraically independent over $K$, the ring $K\left[ \bar{a}\right] \left[ X\right]$ is UFD, the polynomials $f_0,\dotsc , f_N$ are uniquely determined and saying that they have minimal degree is the same as saying that they are coprime.

Assume that the claim of the Lemma does not hold, i. e. that $f'=0$. Using this and the chain rule we can do the following calculations:
\begin{align*}
0 &= \partial \left( f \left( b \right) \right) = \partial \left( \sum_{i=0}^N f_i\left( \bar{a}\right) b^i \right) = \sum_{i=0}^N \partial \left( f_i\left( \bar{a} \right) \right) b^{qi}  \\
&= \sum_{i=0}^N \left( f_i^\partial\left( \bar{a}^q  \right) + \sum_{j=0}^{m-1} \frac{\partial f_i}{\partial X_j} \left( \bar{a}\right)^q \partial^{j+1} \left( a \right)\right) b^{qi} \\
&= \sum_{i=0}^N \left( f_i^\partial\left( \bar{a}^q  \right) + \sum_{j=0}^{m-2} \frac{\partial f_i}{\partial X_j} \left( \bar{a}\right)^q \partial^{j+1} \left( a \right)\right) b^{qi} + \sum_{i=0}^N  \frac{\partial f_i}{\partial X_{m-1}} \left( \bar{a}\right)^q \partial^{m} \left( a \right) b^{qi} \\
&=\sum_{i=0}^N \left( f_i^\partial\left( \bar{a}^q  \right) + \sum_{j=0}^{m-2} \frac{\partial f_i}{\partial X_j} \left( \bar{a}\right)^q \partial^{j+1} \left( a \right)\right) b^{qi} + \sum_{i=0}^N  \frac{\partial f_i}{\partial X_{m-1}} \left( \bar{a}\right)^q  b^{qi+1}.
\end{align*}
Let us define $g\in K\left[ \bar{a}\right] \left[ X \right]$ by the following formula:
$$g\left( X \right) = \sum_{i=0}^N \left( f_i^\partial\left( \bar{a}^q  \right) + \sum_{j=0}^{m-2} \frac{\partial f_i}{\partial X_j} \left( \bar{a}\right)^q \partial^{j+1} \left( a \right)\right) X^{qi} + \sum_{i=0}^N  \frac{\partial f_i}{\partial X_{m-1}} \left( \bar{a}\right)^q  X^{qi+1}.$$
Since the polynomial $g$ vanishes at $b$ we have that, using Gauss Lemma, $g$ is divisible by $f$ in the ring $K\left[ \bar{a}\right] \left[ X \right]$ (here we used the fact that the coefficients of $f$ are coprime). Therefore, since $f'=0$, also $g'$ is divisible by $f$, hence $g'\left( b \right) = 0$. A direct calculation shows that
$$g'\left( X \right) = \sum_{i=0}^N \frac{\partial f_i}{\partial X_{m-1}} \left( \bar{a}\right)^q  X^{qi} = \left( \sum_{i=0}^N \frac{\partial f_i}{\partial X_{m-1}} \left( \bar{a}\right) X^{i}\right)^q,$$
thus $\sum_{i=0}^N\frac{\partial f_i}{\partial X_{m-1}} \left( \bar{a}\right)  b^{i} = 0$. By the minimality assumption on $f$, for any $i$ we have $\frac{\partial f_i}{\partial X_{m-1}} \left( \bar{a}\right) = 0 $ and, since $\bar{a}$ is algebraically independent over $K$, also $\frac{\partial f_i}{\partial X_{m-1}} = 0$.

By repeating analogous calculations we may also show that $\frac{\partial f_i}{\partial X_{j}} = 0$ for any $i, j$ - as an example, we will prove this for $j=m-2$. Recall that $\partial^2=\partial\circ \partial$ is a derivation of $x\mapsto x^{q^2}$. Using the fact $\frac{\partial f_i}{\partial X_{m-1}} = 0$ and the same identities as previously, we arrive at
$$0=\partial^2\left( f \left( b\right) \right) = \sum_{i=0}^N \left( f_i^{\partial^2}\left( \bar{a}^{q^2}  \right) + \sum_{j=0}^{m-3} \frac{\partial f_i}{\partial X_j} \left( \bar{a}\right)^{q^2} \partial^{j+2} \left( a \right)\right) b^{q^2i} + \sum_{i=0}^N  \frac{\partial f_i}{\partial X_{m-2}} \left( \bar{a}\right)^{q^2}  b^{q^2i+1}.$$
Let $h\in K\left[ \bar{a}\right] \left[ X \right]$ play the role of $g$ above. As previously, $h'\left( b \right)=0$ and direct calculations show that
$$h'\left( X \right) = \sum_{i=0}^N  \frac{\partial f_i}{\partial X_{m-2}} \left( \bar{a}\right)^{q^2}  X^{q^2i} = \left(\sum_{i=0}^N  \frac{\partial f_i}{\partial X_{m-2}} \left( \bar{a}\right)  X^{i} \right)^{q^2},$$
hence $\sum_{i=0}^N  \frac{\partial f_i}{\partial X_{m-2}} \left( \bar{a}\right)  b^{i} = 0$, thus by the minimality assumptions on $f$ and the independence of $\bar{a}$ we get that $\frac{\partial f_i}{\partial X_{m-2}}=0$ for any $i$.

Since $\frac{\partial f_i}{\partial X_{j}} = 0$ for any $i, j$, we get that every $f_i$ is a polynomial in $X_0^p,\dotsc , X_{m-1}^p$. Consider the equality
$$\sum_{i=0}^N f_i\left( \bar{a} \right) b^i = 0.$$
Since $f_i=0$ if $i$ is not divisible by $p$ and every $f_i$ is a polynomial in $X_0^p,\dotsc , X_{m-1}^p$, this equality expresses the linear dependence  over $K$ of $p$th powers of some elements of $L$, namely some monomials in $\bar{a}$ and $b$ whose degree in $b$ is not greater than $N/p$. Since $K\subseteq L$ is an separable extension, already this monomials must be linearly dependent over $K$. Because of the algebraic independence of $\bar{a}$ over $K$, this dependence relation must contain $b$. Thus we get that $b$ satisfies some algebraic relation over $K\left[ \bar{a}\right]$ of degree smaller than $N$, contrary to the minimality of $N$.
\end{proof}

\begin{lemma}
\label{simpler}
A strict $\fr^n$-differential field $K$ is 1-existentially closed in the language $L^\partial$ if and only if the following condition holds: if $f, g \in K\{ X\}$ and there is some $\fr^n$-differential extension $K\subseteq L$ such that $L \models \left(\exists x\right) \left( f\left( x \right) = 0, g\left( x \right)\neq 0\right)$, then $K \models \left(\exists x\right) \left( f\left( x \right) = 0, g\left( x \right)\neq 0\right)$.
\end{lemma}
\begin{proof}
The implication $\left( \Longrightarrow \right)$ is immediate. For $\left( \Longleftarrow \right)$ assume that $K$ satisfies the condition from the statement of the Lemma. Let $\phi \left( x \right)$ be a quantifier free $L^\partial$-formula (in one variable, with parameters from $K$) for which there is some extension $K\subseteq L$ such that $L\models \left(\exists x\right) \phi\left( x \right)$. We may assume that this formula is of the form
$$f_1\left( x \right) =  0 \wedge \dotsc \wedge  f_m\left( x \right) = 0 \wedge g\left( x \right) \neq 0$$
for some differential polynomials $f_1,\dotsc, f_m, g \in K\{ X\}$. Take $N$ such that $m<p^N$ and pick some $t\in K \setminus K^p$. By Lemma \ref{lindisj}, for any extension $K\subseteq L$ in the language $L^\partial$ we have $t\in L\setminus L^p$, thus we have
$$L\models \left(\exists x\right)\left( f_1\left( x \right) =  0 \wedge \dotsc \wedge  f_m\left( x \right) = 0 \wedge g\left( x \right) \neq 0\right) \Longleftrightarrow L \models \left(\exists x\right) \left( f\left( x \right) = 0 \wedge g\left( x \right)\neq 0\right),$$
where $f\left( x \right) = f_1\left( x \right)^{p^N}+t f_2\left( x \right)^{p^N}+\dotsc +t^{m-1}f_m\left( x \right)^{p^N}$.  Since $L\models \left(\exists x\right) \phi\left( x \right)$, we have that $L \models \left(\exists x\right) \left( f\left( x \right) = 0\wedge g\left( x \right)\neq 0\right)$, thus $K \models \left(\exists x\right) \left( f\left( x \right) = 0\wedge g\left( x \right)\neq 0\right)$, hence also $K\models \left(\exists x\right) \phi\left( x \right)$. Therefore, $K$ is 1-existentially closed.
\end{proof}

We are now ready to prove our main theorem. Recall that the theory $T_{\operatorname{Wood}}$ was defined in the Introduction.

\begin{theorem}
\label{main}
The theory $T_{\operatorname{Wood}}$ has the same models as $\fr^n - \dcf_p$.
\end{theorem}
\begin{proof}
Since the theory $\fr^n - \dcf_p$ contains $T_{\operatorname{Wood}}$ (see \cite{DerFro}, after Question 3), every model of $\fr^n - \dcf_p$ is a model of $T_{\operatorname{Wood}}$.

For the other direction, note that by Corollary \ref{ec} it is enough to  prove that models of $T_{\operatorname{Wood}}$ are 1-existentially closed $\fr^n$-differential fields in the language $L^\partial$. In order to prove this, we will verify the condition from Lemma \ref{simpler}. Assume that this condition is not verified, i. e. there are $f, g \in K\{ X\}$ and some $\fr^n$-differential extension $K\subseteq L$ such that $L \models \left(\exists x\right) \left( f\left( x \right) = 0, g\left( x \right)\neq 0\right)$, but $K \not\models \left(\exists x\right) \left( f\left( x \right) = 0, g\left( x \right)\neq 0\right)$. Among all such $f, g$ pick a pair which is minimal in the following sense: $f$ has minimal order, minimal degree in the highest variable and $g$ has (for this $f$) minimal order. Note that $f\neq 0$, since $T_{\operatorname{Wood}}$ proves that for any non-zero $g\in K\{ X\}$ there is some $a\in K$ such that $g\left( a \right) \neq 0$.

Let $L$ be an $\fr^n$-extension of $K$ and let $a\in L$ be such that $f\left( a \right) = 0, g\left( a \right)\neq 0$. Let $m$ be the maximal numbers such that the tuple $\partial^{<m}\left( a \right)$ is algebraically independent over $K$. Let $\tilde{F}\left( X \right)$ be the minimal polynomial of $\partial^m \left( a \right)$ over $K\left[ \partial^{<m}\left( a \right)\right]$, in the same sense as in the proof of Lemma \ref{sep}. Let $\tilde{f}$ be the differential polynomial obtained from $\tilde{F}\left( X \right)$ by replacing the appearances of $\partial^{<m} \left( a \right)$ in $\tilde{F}$ by $\left( X, \dotsc , X^{\left( m-1 \right)}\right)$. By Lemma \ref{sep} $f'\neq 0$ and, as in the proof of this Lemma, for any $k>m$ we have $\partial^{k} \left( a \right)\in K\left( \partial^{\le m }  \left( a \right) \right)$ and actually there is a polynomial $P_k\in K\left[ X_0,\dotsc , X_m \right]$, depending only on $\tilde{f}$, such that
\begin{equation*}
\partial^{k} \left( a \right)=\frac{P_k\left( \partial^{\le m } \left( a \right)\right) }{\tilde{f}'\left( \partial^m \left( a \right) \right)^{q^{k-m}}}. \tag{$\ast$}
\end{equation*}
Let $f_0$ be the differential rational function obtained from the differential polynomial $f$ by replacing $X^{\left( k \right)}$ by 
$$\frac{P_k\left( X, X', \dotsc, X^{\left( m \right)} \right)}{\tilde{f}'\left( X^{\left( m \right)} \right)^{q^{k-m}}}$$
for $k>m$. We see that
$$f_0 = \frac{f_1}{\tilde{f}'\left( X^{\left( m \right)} \right)^{N}}$$
for some $N>0$ and some differential polynomial $f_1\in K\left\{ X \right\}$ of order at most $m$. By the equality $\left( \ast \right)$ we have that $f_0\left( a \right) = f\left( a \right)  = 0$. Thus $f_1 = f_2 \cdot \tilde{f}$ for some $f_2\in K \left\{ X \right\}$.

\vspace{0.5cm}

\textbf{Claim 1.} If $b\in L$ is a solution of the system $\tilde{f} \left( x \right) = 0, \  \tilde{f}'\left( x \right) \cdot g\left( x \right) \neq 0$, then it is also a solution of $f\left( x \right) = 0, \ g\left( x \right) \neq 0$.

\textit{Proof of Claim 1.: }
Let $b\in L$ be a solution of the system $\tilde{f} \left( x \right) = 0, \  \tilde{f}'\left( x \right) \cdot g\left( x \right) \neq 0$. Surely $g\left( b \right) \neq 0$, so we need only to show that $f\left( b \right) = 0$. Since $\partial^{<m}\left( a \right)$ is algebraically independent over $K$, by the formula $\left( \ast \right)$ we have that 
$$\partial^{k} \left( b \right)=\frac{P_k\left( \partial^{\le m } \left( b \right)\right) }{\tilde{f}'\left( \partial^m \left( b \right) \right)^{q^{k-m}}}$$
for $k>m$. Thus, by the definition of $f_0$ we have $f_0\left( b \right) = f \left( b \right)$. On the other hand
$$f_0\left( b\right)  = \frac{f_1\left( b\right) }{\tilde{f}'\left( \partial^m\left( b \right) \right)^N} = \frac{f_2 \left( b\right ) \cdot \tilde{f}\left( b\right ) }{\tilde{f}'\left( \partial^m\left( b \right) \right)^N}  = 0,$$
since $\tilde{f} \left( b\right ) = 0$. Thus  $f\left( b \right) = f_0\left( b \right) = 0$, as desired.
\qed
\vspace{0.5cm}

Therefore, by the minimality assumptions on $f$ and $g$, we have that the order of the original $f$ is $m$, $f$ is separable as a polynomial in the variable $X^{(m)}$. Reasoning as in Claim 1. we may assume that the order of $g$ is at most $m$. If $g$ would have order strictly smaller than $m$, then $K\models \left( \exists x \right) \left( f\left( x\right) = 0 , g\left( x\right) \neq 0 \right)$, as $K$ satisfies the Wood axioms, contrary to the assumptions. Thus $g$ has order precisely $m$.

Denote by $F, G\in K\left( \partial^{<m} \left( a \right) \right) \left[ X^{(m)} \right]$ the polynomials obtained from $f, g$ by replacing $\partial^{<m} \left( a \right)$ by $\left( X, \dotsc , X^{\left( m-1 \right)}\right)$. 

\vspace{0.5cm}
\textbf{Claim 2.} The element $F$ and $G$ of $K\left( \partial^{<m} \left( a \right) \right) \left[ X^{(m)} \right]$ are coprime.

\textit{Proof of Claim 2.: }
Denote the greatest common divisor of $F$ and $G$ in $K\left( \partial^{<m} \left( a \right) \right) \left[ X^{(m)} \right]$ by $\tilde{G}$ and assume it is not invertible, i. e. not an element of $K\left( \partial^{<m} \left( a \right) \right)$.  By multiplying $\tilde{G}$ by some non-zero element of $K\left[ \partial^{<m} \left( a \right) \right]$ we may assume that  $\tilde{G}\in K\left[ \partial^{<m} \left( a \right) \right] \left[ X^{(m)} \right]$. By replacing the appearances of $\partial^{<m} \left( a \right)$ in $\tilde{G}$ by $\left( X, \dotsc , X^{\left( m-1 \right)}\right)$ obtain a differential polynomial $\tilde{g}\in K\left\{ X \right\}$. Since $\tilde{G}$ divides of $F, G$ and the tuple $\partial^{<m} \left( a \right)$ is algebraically independent over $K$, we get that $\tilde{g}$ divides $f$ and $g$ in the ring $K\left\{ X \right\}$. Let $\tilde{f}$ be the quotient of $f$ by $\tilde{g}$. Then
$$K \models \left(\exists x\right) \left( f\left( x \right) = 0\wedge g\left( x \right)\neq 0\right) \Longleftrightarrow K \models \left(\exists x\right) \left( \tilde{f}\left( x \right) = 0\wedge g\left( x \right)\neq 0\right)$$
but $\tilde{f}$ has smaller degree in $X^{\left( m \right)}$ than $f$, which contradicts the minimality assumptions on $f$ and $g$.
\qed
\vspace{0.5cm}

Using Claim 2. and Euclidean division in the ring $K\left( \partial^{<m} \left( a \right) \right) \left[ X^{(m)} \right]$ we can find some $P_0, Q_0 \in K\left( \partial^{<m} \left( a \right) \right) \left[ X^{(m)} \right]$ such that
$$P_0\left( X^{(m)} \right)F\left( X^{(m)} \right)+Q_0\left( X^{(m)} \right)G\left( X ^{(m)}\right) = 1$$
By multiplying this equality by some non-zero element of $K\left[ \partial^{<m} \left( a \right) \right]$ we get that there are some non-zero $P, Q \in K\left[ \partial^{<m} \left( a \right) \right] \left[ X^{(m)} \right]$ such that $\tilde{G}:=PF+QG\in K\left[ \partial^{<m} \left( a \right) \right]\setminus \left\{ 0 \right\}$. We again replace all appearances of $\partial^{<m} \left( a \right)$ in $\tilde{G}$ by $\left( X, \dotsc , X^{\left( m-1 \right)}\right)$ and obtain that there are non-zero differential polynomials $p, q\in K\left\{ X \right\}$ such that $\tilde{g}\left( X \right) := p\left( X \right)f\left( X \right)+q\left( X \right)g\left( X \right)$ is non-zero  and of order smaller than $m$.

The differential polynomial $f$ has order $m$  and is separable in $X^{\left( m \right)}$, and $g$ is non-zero, of order smaller that $m$. Thus, the system
$$f\left( x \right) = 0, \ \tilde{g}\left( x \right) \neq 0$$
has a solution in $K$, by the Wood axioms. But, since $\tilde{g}\left( X \right) = p\left( X \right)f\left( X \right)+q\left( X \right)g\left( X \right)$, any solution of this system is also a solution of the original system
$$f\left( x \right) = 0, \ g\left( x \right) \neq 0,$$
which had no solution in $K$ - a contradiction.
\end{proof}

\begin{remark}
\label{geo}
Using Theorem \ref{main} we can now give a positive answer to Question 4 from \cite{DerFro}. We also slightly improve on the form of the geometric axioms suggested there. Let $\left( K, \partial \right)$ be a $\fr^n$-differential field. We work in some big ambient algebraically closed field $\Omega$ containing $K$. For a $K$-variety $V$ we recall the definition of the ``twisted Frobenius tangent bundle of $V$'' (denoted $V^{\left( 1 \right)}$) from \cite{DerFro}. Suppose $V$ is given as the zero locus of an ideal $I\trianglelefteq K\left[ \bar{X} \right]$, where $\bar{X}=\left( X_1,\dotsc , X_m \right)$ is a tuple of variables. Let $\bar{X}'=\left( X'_1,\dotsc , X'_m \right)$ be a new tuple of variables. For $f\in K\left[ \bar{X} \right]$ we set
$$\partial\left( f \right) \left( \bar{X}, \bar{X}' \right) = f^\partial\left( \bar{X}^q \right) + \sum_{i=1}^m \frac{\partial f}{\partial X_i}\left( \bar{X} \right)^q X_i'.$$
We define $V^{\left( 1 \right)}$ as the set of zeroes of the ideal $\left( I, \partial\left( I \right) \right) \trianglelefteq K\left[ \bar{X}, \bar{X}' \right]$. If we took in the above formulas $\partial = 0, q= 0$, then the resulting $V^{\left( 1 \right)}$ would be the usual tangent bundle of $V$. Note that for any $a\in V\left( K \right)$ we have $\left( a, \partial\left( a \right) \right)\in V$ and thus we have a projection map $V^{\left( 1 \right)}\to V$.

We are now able to state the new geometric axioms for $\fr^n-\dcf_p$:

\begin{quote}
    Suppose $V, W$ are $K$-irreducible $K$-varieties and $W\subseteq V^ {\left( 1 \right)}$. If the projection map $W\to V$ is separable, then there is some $a\in V\left( K \right)$ such that $\left( a, \partial\left( a \right) \right)\in W$.
\end{quote}
It is standard, that this is expressible by a scheme of first-order conditions, as explained in \cite{DerFro}. This looks a bit different than what is suggested in Question 4 from \cite{DerFro}, so we explain the changes. Literally, Question 4 suggest the following axioms
\begin{quote}
    Assume $K$ is strict. Suppose $V, W$ are $K$-irreducible $K$-varieties and $W\subseteq V^{\left( 1 \right)}$. If the projection map $W\to V$ is dominant and separable, then there is some $a\in V\left( K \right)$ such that $\left( a, \partial\left( a \right) \right)\in W\setminus X$.
\end{quote}
The word ``dominant'' is redundant, as separable morphisms are by definition dominant. The variety $X$ can be removed by replacing $W$ and $V$ by some higher dimensional variety. Finally, strictness of $K$ follows from the latter part of the axioms, as follows. Assume $K$ is a $\fr^n$-differential field satisfying the above axioms, but without the sentence ``$K$ is strict''. We will prove that $K$ is strict. Assume that is not the case and let $c\in K^\partial \setminus K^p $ be a constant, which is not a $p-$th power. Define $V:=\left\{ c^{1/p} \right\}$, i. e. $V$ is the zero locus of the irreducible polynomial $f\left( X \right) = X^p-c$. Thus, the twisted tangent bundle $V^{(1)}$ is the zero locus of the polynomial
$$\partial\left( f \right) \left ( X, X' \right) = f^\partial \left( X^q \right) + f'\left( X \right)^q X' = 0,$$
i. e. $V^{(1)} = \left\{ c^{1/p} \right\}\times \mathbb{A}^1$. Set $W:=\left\{ \left( c^{1/p}, 0 \right) \right\}$. Clearly $W\subseteq V^{(1)}$ and the projection map $W\to V$ is separable, therefore by the geometric axioms there exists some $a\in V\left( K \right)$ such that $\left( a, \partial\left( a \right)\right) \in W\left( K \right)$. But $a\in V\left( K \right)$ means that $a\in K$ and $a^p=c$, contrary to the assumption that $c\not\in K^p$.

Since the axioms described above are contained in the original axioms of $\fr^n-\dcf_p$ and they contain the Wood axioms (as explained in \cite{DerFro} above Question 4), we get a new geometric axiomatization of $\fr^n-\dcf_p$.

\end{remark}
\begin{remark}
For the sake of completeness we point out, that Question 6 from \cite{DerFro} has an obvious negative answer. It asks whether $\left( K, \partial^2 \right)$ is a model of $\fr^{2n}-\dcf_p$, provided that $\left( K, \partial \right)$ is a model of $\fr^{n}-\dcf_p$. The answer is negative, since the constants of $\left( K, \partial^2 \right)$ are strictly bigger than $K^p=K^\partial$, so $\left( K, \partial^2 \right)$  is not even strict.
\end{remark}

\section*{Acknowledgment}
I would like to thank Piotr Kowalski for his careful reading of this paper and many valuable comments.
\bibliographystyle{plain}
\bibliography{main}

\begin{thebibliography}{1}

\bibitem{BHKK}
\"{O}zlem {Beyarslan}, Daniel {Hoffmann}, Moshe {Kamensky}, and Piotr
  {Kowalski}.
\newblock Model theory of fields with free operators in positive
  characteristic.
\newblock {\em Transactions of the American Mathematical Society},
  372(8):5991--6016, 2019.

\bibitem{ChGeneric}
Zo{\'e} Chatzidakis.
\newblock Generic automorphisms of separably closed fields.
\newblock {\em Illinois Journal of Mathematics}, 45:693--733, 2001.

\bibitem{beta-op}
Jakub {Gogolok} and Piotr {Kowalski}.
\newblock {Operators coming from ring schemes}.
\newblock {\em arXiv e-prints}, page arXiv:2008.08443, August 2020.

\bibitem{DerFro}
Piotr Kowalski.
\newblock Derivations of the {F}robenius map.
\newblock {\em Journal of Symbolic Logic}, 70(1):99--110, 2005.

\bibitem{PiercePillay}
David Pierce and Anand Pillay.
\newblock A note on the axioms for differentially closed fields of
  characteristic zero.
\newblock {\em Journal of Algebra}, 204:108--115, 1998.

\bibitem{pogudin}
Gleb Pogudin.
\newblock A primitive element theorem for fields with commuting derivations and
  automorphisms.
\newblock {\em Selecta Mathematica}, 25(4), 2019.

\bibitem{Seidenberg}
Abraham Seidenberg.
\newblock Some basic theorems in differential algebra (characteristic $p$,
  arbitrary).
\newblock {\em Transactions of the American Mathematical Society},
  73(1):174--190, 1952.

\bibitem{Wo1}
Carol Wood.
\newblock The model theory of differential fields revisited.
\newblock {\em Israel Journal of Mathematics}, 25:331--352, 1976.

\end{thebibliography}

\end{document}